\documentclass[12pt, reqno]{amsart}
\usepackage{amsmath, amsthm, amscd, amsfonts, amssymb, graphicx, color}
\usepackage[bookmarksnumbered, colorlinks, plainpages]{hyperref}
\hypersetup{colorlinks=true,linkcolor=red, anchorcolor=green, citecolor=cyan, urlcolor=red, filecolor=magenta, pdftoolbar=true}

\newtheorem{pro}{Proposition}[section]
 \newtheorem{cor}[pro]{Corollary}
 \newtheorem{lem}[pro]{Lemma}

  \newtheorem{defi}[pro]{Definition}
 \newtheorem{prop}[pro]{Proposition}
 \newtheorem{thm}[pro]{Theorem}

 \newcommand{\h}{\mathcal H}
   
  \newcommand{\hb}{\mathbb H}
  \newcommand{\ff}{\mathbf F}

 \numberwithin{equation}{section}

\begin{document}

\setcounter{page}{1}

\title[Inequalities-Equalities Concerning the continuous g-Fusion Frame ]{On Some Inequalities-Equalities Concerning the continuous generalized Fusion Frame in Hilbert spaces}

\author[NADIA ASSILA, SAMIR KABBAJ, Ouafaa Bouftouh  and Chaimae Mezzat ]{NADIA ASSILA$^{1*}$, SAMIR KABBAJ$^{2}$, Ouafaa Bouftouh$^{3}$\MakeLowercase{and} Chaimae Mezzat$^{4}$}

\address{$^{1}$Department of Mathematics, University of Ibn Tofail, B.P. 133, Kenitra, Morocco}
\email{\textcolor[rgb]{0.00,0.00,0.84}{samkabbaj@yahoo.fr}}

\address{$^{2}$Department of Mathematics, University of Ibn Tofail, B.P. 133, Kenitra, Morocco}
\email{\textcolor[rgb]{0.00,0.00,0.84}{nadiyassila@gmail.com}}

\address{$^{3}$Department of Mathematics, University of Ibn Tofail, B.P. 133, Kenitra, Morocco}
\email{\textcolor[rgb]{0.00,0.00,0.84}{bouftouh2012@gmail.com}}

\address{$^{3}$Department of Mathematics, University of Ibn Tofail, B.P. 133, Kenitra, Morocco}
\email{\textcolor[rgb]{0.00,0.00,0.84}{chaimae.mezzat@uit.ac.ma}}

\address{ Laboratoire des Equations aux Dérivées Partielles\\ Algèbre et Géométrie Spectrale.
  }


\date{02/10/2021; 
\newline \indent $^{*}$Corresponding author}

\begin{abstract}
continuous generalized fusion frame theory was recently introduced by Rahimi and al. Several equalities and inequalities have been obtained for frame, fusion generalized fusion frame, among others. In the present paper, we continue and extend these results to obtain  some important identities and inequalities in the case of  continuous generalized fusion frame, Parceval continuous generalized fusion frame, $ \lambda$-tight continuous generalized fusion frame. Moreover, we obtain some new inequalities for the alternate dual continuous generalized fusion frame. Finally, we obtain frame operator of a pair of Bessel continuous generalized fusion mapping and we derive some results about resolution of identity.
\end{abstract}
  \maketitle
\section{Introduction} 
Frames are among the most intensively studied and best understood of all classes of overcomplete basis: one can represent each element in the vector space via a frame.
In last few decades this notion has attracted much attention because of their practical applications in many areas such as coding and communications, filter bank theory, and widely used in signal and image processing, among others.
More recently, Sadri and al. are combined two types of frames which lead to a new concept in the theory of frames $\ll\, generalized \,fusion\, frame \,\gg$.\\

One of the most important inequality and useful identity  were found by Balan and al.(2006) as follows: 
\begin{thm}\label{th: Balan identitie}(Balan and al. 2007) Let $\{f_i\}_{i\in I}$ be a Parceval frame for $\h$. Then 
\begin{equation}\label{eq: Balan identitie}
\sum_{i\in J}\vert\langle f, f_i\rangle\vert^2-\left\Vert \sum_{i\in J}\langle f, f_i\rangle f_i\right\Vert^2\\=
\sum_{i\in J^c}\vert\langle f, f_i\rangle\vert^2- \left\Vert \sum_{i\in J^c}\langle f, f_i\rangle f_i\right\Vert^2.
\end{equation}
\end{thm}
which is particulary used in the study of signal processing. Inspired by \ref{th: Balan identitie} in 2006 Zhu and Wu generalized this inequality to an alternate dual frame.
\begin{thm}\label{th: Zhu and Wu identitie}(Zhu and Wu 2006) Let $\{f_i\}_{i\in I}$ be a frame for $\h$ and  $\{g_i\}_{i\in I}$ is an alternate dual frame of $\{f_i\}_{i\in I}$. Then for any $J\in I$ and $f\in \h$ we have 
\begin{equation}\label{eq: Zhu and Wu identitie}
\left(\sum_{i\in J}\langle f, g_i\rangle \overline{\langle f,f_i\rangle}\right)-\left\Vert \sum_{i\in J}\langle f, g_i\rangle f_i\right\Vert^2\\=
\left(\sum_{i\in J^c}\langle f, g_i\rangle \overline{\langle f,f_i\rangle}\right)- \left\Vert \sum_{i\in J^c}\langle f, g_i\rangle f_i\right\Vert^2.
\end{equation}
\end{thm}
Later on, this inequality (\ref{eq: Balan identitie}) has motivated a large number of authors such as Guo and al.(2016) Zhang and Li (2016), for various other inequalities related to this inequality, we refer the readers to (\cite{operatoroperation1}, \cite{Blana06}, \cite{Balan15}, \cite{Gavruta06}, \cite{ahmadi20}, \cite{Lisun08}).

Motivated by the aforementioned works, we aim to  extend and improve identities and inequalities in the case of  continuous generalized fusion frame, Parceval continuous generalized fusion frame, $ \lambda$-tight continuous generalized fusion frame and continuous generalized fusion pairs.
\section{Preliminaries}
\subsection{Background}
Throughout this paper, we adopt the following notations: $\h$ will be a Hilbert space, $\mathcal{B}(\h)$ the algebra of all bounded linear operators on $\h$, $Id_{\h}$ the identity operator on $\h$, and $\hb$ the collection of all closed subspace of $\h$. Also, $(X, \mu)$ will be a measure space, and $\mu\,:\,\, X \longrightarrow [0,+\infty)$ will be a measurable mapping such that $\mu\neq 0\,\, a.e.$.
\begin{lem} \cite{operatoroperation2}\label{lem: operatoroperation1} Let $U\in\mathcal{\h}$ be a self-adjoint and $T=aU^2+bU+c Id_{\h}$ such that $a, b, c \in \mathbb{R}$, then the following statements hold:
\begin{itemize}
\item[i)] If $a>0$, then
$$
\inf _{\|f\|=1}\langle Tf, f\rangle \geq \frac{4 a c-b^{2}}{4 a} .
$$
\item[ii)] If $a<0$, then
$$
\sup _{\|f\|=1}\langle T f, f\rangle \leq \frac{4 a c-b^{2}}{4 a} .
$$
\end{itemize}
\end{lem}   
\begin{lem} \cite{operatoroperation1}\label{lem: operatoroperation2}  If $T_1, T_2$ are operators on $\h$ satisfying $T_1+ T_2=Id_{\h}$, then $T_1+ T_2= T_1^2+ T_2^2.$
 \end{lem}
\begin{lem}\cite{Ga07}\label{lem: Projection-operator}
Let $V \subseteq \mathcal{H}$ be a closed subspace, and $T$ be a linear bounded operator on
$\mathcal{H}$. Then
\begin{align}
\label{lem: pi-T}
\pi_{V} T^*=\pi_ V T^*\pi_{\overline{{T V}}}.
\end{align}
If $T$ is a unitary (i.e. $T^*T = T T^*=Id_{\mathcal{H}}$, then
\begin{align}
\label{lem: piTunitaire}
\pi_{\overline{{TV}}} T=T\pi_V .
\end{align} \end{lem}
The following definition is a generalization continuous version of fusion frames proposed and defined by Faroughi and Ahmadi \cite{CFUSION FRAMES} as follows: 
\begin{defi}(see \cite{CFUSION FRAMES}) let $\mathbf{F}:\quad X\longrightarrow \hb$ be such that for each $f\in\h$, the mapping $x\longmapsto\pi_{F(x)}(f)$ is measurable $($i.e. is weakly measurable$) $ and let $\{\h_x\}_{x\in X}$ be a collection of Hilbert spaces. For each $x\in X$, suppose that $\Lambda_x\in\mathcal{B}( \mathbf{F}(x),\h_x)$ and put
	\begin{align*}
	\Lambda=\{ \Lambda_x\in \mathcal{B}( \mathbf{F}(x),\h_x):\,\, x\in X\}.
	\end{align*}
Then $(\Lambda, \ff,\omega)$ is a continuous $g$-fusion frame for $\h$ if there exist $0 <A \leq B<\infty$ such that for all $f\in\h$
\begin{equation} \label{def: cgff}
A \Vert f\Vert^2 \leq \int _{X_1}\omega^2(x) \Vert \Lambda_x  \pi_{F(x)}(f)\Vert^2 d\mu(x) \leq B \Vert f\Vert^2 , 
\end{equation}
Where $\pi_{F(x)}$  is the orthogonale projection of $\h$ onto subspace $F(x)$.
\end{defi}
$(\Lambda, \ff,\omega)$ is  called a tight continuous $g$-fusion frame for $\h$ if $A=B$, and Parseval if $A=B=1$.\\
$(\Lambda, \ff,\omega)$ is called a Bessel continuous $g$-fusion frame for $\h$ if the right inequality holds.\\
Let $\mathcal{K}=\oplus_{x\in X}\mathcal{K}_x$ and $L^2(X,\mathcal{K})$ be a collection of all measurable function $\varphi:\quad X\longrightarrow\mathcal{K}$ such that for each $x\in X$ $ \varphi(x) \in \mathcal{K}_x$ and \begin{align*}
\int_X \Vert \varphi(x) \Vert ^2<\infty.
\end{align*}
The synthesis operator is defined weakly as follows (for more details refer to \cite{AK20}):
\begin{eqnarray*}
T_{\mathbf{F},\Lambda}:\quad L^2(X,\mathcal{K}) \longrightarrow \h,\\
\langle T_{\mathbf{F},\Lambda}(\varphi),f\rangle=\int_X \omega(x)\langle \Lambda^*_x(\varphi(x)),h\rangle d\mu(x),
\end{eqnarray*}
 where $ \varphi \in L^2(X,\mathcal{K}) $ and $h\in \h $. It is obvious that $T_{\mathbf{F},\Lambda}$ is linear and by Remark $(1.6)$ in \cite{CFUSION FRAMES}, $T_{\mathbf{F},\Lambda}$ is a bounded linear operator. Its adjoint, that is called analysis operator is:
 \begin{eqnarray*}
 T_{\mathbf{F},\Lambda}^{\ast}:\quad  \h \longrightarrow L^2(X,\mathcal{K},\\
 T_{\mathbf{F},\Lambda}^{\ast} =\omega(.) \Lambda^*_{(.)}\pi_{F(.)}. 
 \end{eqnarray*}
\begin{align}
S_{\mathbf{F},\Lambda}(f)= T_{\mathbf{F},\Lambda} T_{\mathbf{F},\Lambda}^{\ast}(f) =\int _{X_1}\omega^2(x) \pi_{F(x)}\Lambda^*_x\Lambda_x  \pi_{F(x)}(f) d\mu(x),\quad f\in\h.
\end{align}
$S_{\mathbf{F},\Lambda}$ is a bounded, positive,  self-adjoint and invertible operator. and we have
$$
B^{-1} i d_{\h} \leq S_{\mathbf{F},\Lambda}^{-1} \leq A^{-1} i d_{\h} .
$$
So, we have the following reconstruction formula for any $f \in \h$ :
\begin{equation}\label{eq: reconstruction formula}
f=\int _{X}\omega^2(x)   \pi_{F(x)} \Lambda_{x}^{*} \Lambda_{x} \pi_{F(x)} S_{\mathbf{F},\Lambda}^{-1} (f)d\mu(x)=\int _{X}\omega^2(x)   S_{\mathbf{F},\Lambda}^{-1} \pi_{F(x)} \Lambda_{x}^{*} \Lambda_{x} \pi_{F(x)} (f)d\mu(x).
\end{equation}

\section{Inequalities-Equalities for Parseval continuous generalized fusion frame}
 Let  $(\Lambda, \ff,\omega)$ be a continuous $g$-fusion frame for $\h$ with bounds $A$ and $B$. Denoting its canonical dual continuous $g$-fusion frame by $\tilde{\Lambda}:=\left(S_{\mathbf{F},\Lambda}^{-1} F(x), \Lambda_{x} \pi_F(x) S_{\mathbf{F},\Lambda}^{-1},\omega \right)$. Hence for each $f \in \h$ the reconstruction fomula \ref{eq: reconstruction formula} may be written in the form,
$$
f=\int _{X}\omega^2(x)  \pi_{F(x)} \Lambda_{x}^{*} \tilde{\Lambda}_{x}  \pi_{\tilde{F}(x)} (f) d\mu(x)=\int _{X}\omega^2(x)  \pi_{\tilde{F}(x)}\tilde{\Lambda}_{x}^{*} \Lambda_{x}  \pi_{F(x)}(f)d\mu(x).
$$
where $\tilde{F}(x):=S_{\mathbf{F},\Lambda}^{-1} F(x),\, \tilde{\Lambda}_{x}:=\Lambda_{x} \pi_F(x) S_{\mathbf{F},\Lambda}^{-1}$. Thus, we obtain\begin{equation}\label{eq: operator inverse}
\langle S_{\mathbf{F},\Lambda}^{-1} f,f \rangle= \int _{X_1}\omega^2(x) \Vert \tilde{\Lambda}_{x}  \pi_{\tilde{F}(x)}(f)\Vert^2 d\mu(x).
\end{equation}
For any $X_1\subset X$, we denote  $X_1^c=X\setminus X_1 $, and we define the following operators:
\begin{align}
S_{\mathbf{F},\Lambda}^{X_1}f=\int _{X_1}\omega^2(x) \pi_{\mathbf{F}(x)}\Lambda^*_x \tilde{\Lambda}_x  \pi_{{\tilde{F}(x)}}(f) d\mu(x),\quad f\in\h.
\end{align}

\begin{align}
\mathcal{M}_{\mathbf{F},\Lambda}^{X_1}f &=\int _{X_1}\omega^2(x) \pi_{\mathbf{F}(x)}\Lambda^*_x \Lambda_x  \pi_{F(x)}(f) d\mu(x),\quad f\in\h.
\end{align}
Obviously,   $S_{\mathbf{F},\Lambda}=\mathcal{M}_{\mathbf{F},\Lambda}^{X_1}+\mathcal{M}_{\mathbf{F},\Lambda}^{X_1^c}$, and $\mathcal{M}_{\mathbf{F},\Lambda}^{X_1}$, $\mathcal{M}_{\mathbf{F},\Lambda}^{X_1^c}$ are self-adjoint operators. It is easy to check that  $S_{\mathbf{F},\Lambda}^{X_1}$ is a bounded, linear and positive operator. Again, we have
 \begin{align*}
 S_{\mathbf{F},\Lambda}^{X_1}+S_{\mathbf{F},\Lambda}^{X_1^c}=Id_{\h}.
 \end{align*}
\begin{thm}\label{th: theorem1} Let $f\in\h$, then 
\begin{multline}\label{mult: th1}
		 \int _{X_1}\omega^2(x) \langle \tilde{\Lambda}_{x}  \pi_{\tilde{F}(x)}(f),\Lambda_x  \pi_{F(x)}(f)\rangle d\mu(x)- \Vert S_{\mathbf{F},\Lambda}^{X_1} f\Vert^2 \\=\int _{X_1^c}\omega^2(x) \overline{\langle \tilde{\Lambda}_{x}  \pi_{\tilde{F}(x)}(f),\Lambda_x  \pi_{F(x)}(f)\rangle} d\mu(x)- \Vert S_{\mathbf{F},\Lambda}^{X_1^c} f\Vert^2 .
\end{multline}
\end{thm}
\begin{proof} For any $f\in \h$, we have 
\begin{align*}
 &\int _{X_1}\omega^2(x) \langle \tilde{\Lambda}_{x}  \pi_{\tilde{F}(x)}(f),\Lambda_x  \pi_{F(x)}(f)\rangle d\mu(x)- \Vert S_{\mathbf{F},\Lambda}^{X_1} f\Vert^2 \\&= \langle S_{\mathbf{F},\Lambda}^{X_1}f,f \rangle-\Vert S_{\mathbf{F},\Lambda}^{X_1} f\Vert^2\\&=
 \langle S_{\mathbf{F},\Lambda}^{X_1}f,f \rangle-\langle  (S_{\mathbf{F},\Lambda}^{X_1})^*S_{\mathbf{F},\Lambda}^{X_1} f,f\rangle\\&=
 \langle  (Id_{\h}-S_{\mathbf{F},\Lambda}^{X_1})^*S_{\mathbf{F},\Lambda}^{X_1} f,f\rangle =
\langle  (S_{\mathbf{F},\Lambda}^{X_1^c})^*(Id_{\h}-S_{\mathbf{F},\Lambda}^{X_1^c}) f,f\rangle \\&= \langle ( S_{\mathbf{F},\Lambda}^{X_1^c})^{*}f,f \rangle-\langle  (S_{\mathbf{F},\Lambda}^{X_1^c})^*S_{\mathbf{F},\Lambda}^{X_1^c} f,f\rangle.
\end{align*}
Wich completes the proof.
\end{proof}
Furthermore, if we suppose that $(\Lambda, \ff,\omega)$ is a Parseval continuous generalized fusion frame, so we can  easily obtain the same equality  presented in \cite{operatoroperation2} as  :

\begin{thm} \label{thm: parsevalcgfframe} Assume that $(\Lambda, \ff,\omega)$ is a  Parseval continuous generalized fusion frame for $\h$. Then for $X_1\subset X$ and $f\in \h$. The following results hold:
\begin{multline} 
\int _{X_1}\omega^2(x) \Vert \Lambda_x\pi_{F(x)}(f)\Vert^2 d\mu(x)-\Vert \int _{X_1}\omega^2(x) \pi_{F(x)}\Lambda_x^*\Lambda_x  \pi_{F(x)}(f) d\mu(x)\Vert^2 \\= \int _{X_1^c}\omega^2(x) \Vert \Lambda_x\pi_{F(x)}(f)\Vert^2d\mu(x)-\Vert \int _{X_1^c}\omega^2(x) \pi_{F(x)}\Lambda^*_x\Lambda_x  \pi_{F(x)}(f) d\mu(x)\Vert^2.
\end{multline}
	Morever,\begin{align} 
	\int _{X_1^c}\omega^2(x) \Vert \Lambda_x\pi_{F(x)}(f)\Vert^2d\mu(x)-\Vert \int _{X_1^c}\omega^2(x) \pi_{F(x)}\Lambda^*_x\Lambda_x  \pi_{F(x)}(f) d\mu(x)\Vert^2 \geq\frac{3}{4}\Vert  f\Vert^2.
	\end{align}
\end{thm}
\begin{proof}
Since $(\Lambda, \ff,\omega)$ is a Parseval continuous generalized fusion frame for $\h$, and by using the fact that $S_{\mathbf{F},\Lambda}^{X_1}$ and $S_{\mathbf{F},\Lambda}^{X_1^c}$ are commuting, then for each $f\in\h$, we have
\begin{align*}
&\int _{X_1}\omega^2(x) \Vert \Lambda\pi_{F(x)}(f)\Vert^2 d\mu(x)-\Vert \int _{X_1}\omega^2(x) \pi_{F(x)}\Lambda^*_x\Lambda_x  \pi_{F(x)}(f) d\mu(x)\Vert^2 \\&=\langle (S_{\mathbf{F},\Lambda}^{X_1^c}+(S_{\mathbf{F},\Lambda}^{X_1^c} )^2) f , f \rangle\\&=
\langle ( Id_{\h}-S_{\mathbf{F},\Lambda}^{X}+(S_{\mathbf{F},\Lambda}^{X_1^c} )^2) f , f \rangle.
\end{align*}
Applying Lemma \ref{lem: operatoroperation1} for $a=1$,$b=-1$, and $c=1$, the result follows.
\end{proof}
\begin{cor}\label{cor: parsevalcgfframe} Let $(\Lambda, \ff,\omega)$ is a Parseval continuous generalized fusion frame for $\h$, we have 	 
\begin{multline}
		\frac{1}{2}\Vert  f\Vert^2 \leq \Vert \int _{X_1}\omega^2(x) \pi_{F(x)}\Lambda^*_x\Lambda_x  \pi_{F(x)}(f) d\mu(x)\Vert ^2d\mu(x)\Vert^2\\-\Vert \int _{X_1^c}\omega^2(x) \pi_{F(x)}\Lambda^*_x\Lambda_x  \pi_{F(x)}(f) d\mu(x)\Vert^2 \leq \frac{3}{2}\Vert  f\Vert^2.
	\end{multline}
\begin{multline} 
	\frac{3}{4}\Vert  f\Vert^2 \leq \int _{X_1}\omega^2(x) \Vert \Lambda\pi_{F(x)}(f)\Vert^2d\mu(x)\\-\Vert \int _{X_1^c}\omega^2(x) \pi_{F(x)}\Lambda^*_x\Lambda_x  \pi_{F(x)}(f) d\mu(x)\Vert^2 \leq\Vert  f\Vert^2.
\end{multline}
\end{cor}
\begin{proof}
 Observe that \begin{align*}
 (S_{\mathbf{F},\Lambda}^{X_1})^2+(S_{\mathbf{F},\Lambda}^{X_1^c})^2 &= (S_{\mathbf{F},\Lambda}^{X_1})^2+(S_{\mathbf{F},\Lambda}^{X_1^c} )^2 \\&= 2(S_{\mathbf{F},\Lambda}^{X_1})^2-2S_{\mathbf{F},\Lambda}^{X_1}+Id_{\h}.
 \end{align*}
Applying Lemma 	\ref{lem: operatoroperation1}, we get
\begin{align*}
	(S_{\mathbf{F},\Lambda}^{X_1})^2+(S_{\mathbf{F},\Lambda}^{X_1^c})^2\geq \frac{1}{2}Id_{\h}.
\end{align*}
Since $S_{\mathbf{F},\Lambda}^{X_1}-(S_{\mathbf{F},\Lambda}^{X_1})^2 \geq 0$ and 
\begin{align*}
(S_{\mathbf{F},\Lambda}^{X_1})^2+(S_{\mathbf{F},\Lambda}^{X_1^c})^2 &= 2 (S_{\mathbf{F},\Lambda}^{X_1})^2-2S_{\mathbf{F},\Lambda}^{X_1 } + Id_{\h}\\&=Id_{\h}+ 2 S_{\mathbf{F},\Lambda}^{X_1} -2(S_{\mathbf{F},\Lambda}^{X_1})^2+4 ((S_{\mathbf{F},\Lambda}^{X_1})^2-S_{\mathbf{F},\Lambda}^{X_1}).
\end{align*}
Then, we have \begin{align*}
(S_{\mathbf{F},\Lambda}^{X_1})^2+(S_{\mathbf{F},\Lambda}^{X_1^c})^2 \leq Id_{\h}+ 2 S_{\mathbf{F},\Lambda}^{X_1} -2(S_{\mathbf{F},\Lambda}^{X_1})^2
\end{align*}
Applying again Lemma \ref{lem: operatoroperation1}, we get
\begin{align*}
(S_{\mathbf{F},\Lambda}^{X_1})^2+(S_{\mathbf{F},\Lambda}^{X_1^c})^2 \leq \frac{3}{2}Id_{\h}.
\end{align*}
Thus \begin{align*}
\frac{1}{2} Id_{\h}\leq (S_{\mathbf{F},\Lambda}^{X_1})^2+(S_{\mathbf{F},\Lambda}^{X_1^c})^2 \leq \frac{3}{2}Id_{\h}.
\end{align*}
Next,  observing that \begin{eqnarray*}
S_{\mathbf{F},\Lambda}^{X_1}-(S_{\mathbf{F},\Lambda}^{X_1^c})^2 &=& S_{\mathbf{F},\Lambda}^{X_1}-(Id_{\h}-S_{\mathbf{F},\Lambda}^{X_1 })^2 \\&=&(S_{\mathbf{F},\Lambda}^{X_1 })^2 - S_{\mathbf{F},\Lambda}^{X_1}+ Id_{\h}.
\end{eqnarray*}
And since $S_{\mathbf{F},\Lambda}^{X_1}-(S_{\mathbf{F},\Lambda}^{X_1})^2\geq 0$, implies that 
\begin{align*}
\frac{3}{2} Id_{\h}\leq  S_{\mathbf{F},\Lambda}^{X_1}+(S_{\mathbf{F},\Lambda}^{X_1})^2 \leq  Id_{\h}.
\end{align*} 
by Lemma \ref{lem: operatoroperation1}, for each $f\in\h$, we get
\begin{eqnarray*}
&& \langle (S_{\mathbf{F},\Lambda}^{X_1}-(S_{\mathbf{F},\Lambda}^{X_1})^2)f,f\rangle=\langle  S_{\mathbf{F},\Lambda}^{X_1}f,f\rangle-\langle(S_{\mathbf{F},\Lambda}^{X_1})^2 f,f\rangle \\&=&
  \int _{X_1}\omega^2(x) \Vert \Lambda\pi_{F(x)}(f)\Vert^2d\mu(x) -\Vert \int _{X_1^c}\omega^2(x) \pi_{F(x)}\Lambda^*_x\Lambda_x  \pi_{F(x)}(f) d\mu(x)\Vert^2.
\end{eqnarray*}
The proof is completed.
\end{proof}
\begin{cor}
Let $(\Lambda, \ff,\omega)$ be a  Parseval continuous generalized fusion frame for $\h$. Then 
\begin{equation}
0\leq S_{\mathbf{F},\Lambda}^{X_1}-(S_{\mathbf{F},\Lambda}^{X_1})^2\leq \frac{1}{4}Id_{\h}.
\end{equation}
\end{cor}
\begin{proof} Since $S_{\mathbf{F},\Lambda}^{X_1}S_{\mathbf{F},\Lambda}^{X_1^c}=S_{\mathbf{F},\Lambda}^{X_1^c}S_{\mathbf{F},\Lambda}^{X_1}$, and $S_{\mathbf{F},\Lambda}^{X_1}$, $S_{\mathbf{F},\Lambda}^{X_1^c}$ are positive, self-adjoint operators, it follows that $S_{\mathbf{F},\Lambda}^{X_1}S_{\mathbf{F},\Lambda}^{X_1^c}$ is also positive and self-adjoint. Hence, we have 
$$ 0\leq S_{\mathbf{F},\Lambda}^{X_1}S_{\mathbf{F},\Lambda}^{X_1^c}= S_{\mathbf{F},\Lambda}^{X_1}-(S_{\mathbf{F},\Lambda}^{X_1^c})^2. $$
Applying Lemma \ref{lem: operatoroperation1} yields
 $$ S_{\mathbf{F},\Lambda}^{X_1}-(S_{\mathbf{F},\Lambda}^{X_1})^2\leq \frac{1}{4}Id_{\h}.$$
\end{proof}

Observe that  $S_{\mathbf{F},\Lambda}$ (resp. $S_{\mathbf{F},\Lambda}^{-1}$) is a positive operator in $\mathcal{B}(\h)$, then there exists a unique positive square root $S_{\mathbf{F},\Lambda}^{\frac{1}{2}}$ (resp. $S_{\mathbf{F},\Lambda}^{-\frac{1}{2}}$) which commutes with every operator which commutes with  $S_{\mathbf{F},\Lambda}$ (resp. $S_{\mathbf{F},\Lambda}^{-1}$). Therefore, for each $f\in \h$ we have 
\begin{align*}
f&=S_{\mathbf{F},\Lambda}^{-\frac{1}{2}} S_{\mathbf{F},\Lambda}S_{\mathbf{F},\Lambda}^{-\frac{1}{2}}f\\
&= \int _{X_1}\omega^2(x) S_{\mathbf{F},\Lambda}^{-\frac{1}{2}}\pi_{F(x)}\Lambda_x^*\Lambda_x\pi_{F(x)}  S_{\mathbf{F},\Lambda}^{-\frac{1}{2}} f d\mu(x),
\end{align*}
thus, applying Lemma \ref{lem: Projection-operator} gives that
$$
\begin{aligned}
\|f\|^{2} &=  \langle\int _{X_1}\omega^2(x) S_{\mathbf{F},\Lambda}^{-\frac{1}{2}}\pi_{F(x)}\Lambda_x^*\Lambda_x\pi_{F(x)}  S_{\mathbf{F},\Lambda}^{-\frac{1}{2}} f d\mu(x), f \rangle \\
&=\int _{X_1}\omega^2(x)\left\|\Lambda_{x} \pi_{F(x)} S_{\mathbf{F},\Lambda}^{-\frac{1}{2}} f\right\|^{2} d\mu(x)\\
&=\int _{X_1}\omega^2(x) \left\|\Lambda_{x}\pi_{F(x)} S_{\mathbf{F},\Lambda}^{-\frac{1}{2}} \pi_{S_{\mathbf{F},\Lambda}^{-\frac{1}{2}}{F(x)}} f\right\|^{2}d\mu(x)
\end{aligned}
$$
which means that $(S_{\mathbf{F},\Lambda}^{-\frac{1}{2}}F(x), \Lambda_{x} \pi_{F(x)}S_{\mathbf{F},\Lambda}^{-\frac{1}{2}},\omega  )$ is a Parseval continuous generalized fusion frame. Hence we have the following theorem:
\begin{thm} \label{thm: ParsevalframeS}
Let $(\Lambda, \ff,\omega)$ be a  Parseval continuous generalized fusion frame for $\h$. Then 
	\begin{multline}
	\int _{X_1}\omega^2(x) \Vert \Lambda_x\pi_{F(x)}(f)\Vert^2d\mu(x)- \Vert S_{\mathbf{F},\Lambda}^{-\frac{1}{2}} S_{\mathbf{F},\Lambda}^{X_1} f\Vert^2 \\=\int _{X_1^c}\omega^2(x) \Vert \Lambda_x\pi_{F(x)}(f)\Vert^2d\mu(x)- \Vert S_{\mathbf{F},\Lambda}^{-\frac{1}{2}} S_{\mathbf{F},\Lambda}^{X_1^c} f\Vert^2.
\end{multline}
\end{thm}
\begin{proof}
Assume that $ \chi_x:=\Lambda_x \pi_{F(x)}S_{\mathbf{F},\Lambda}^{-\frac{1}{2}}$ and $ V(x):= S_{\mathbf{F},\Lambda}^{\frac{1}{2}} F(x)$, then by applying the previous result $(S_{\mathbf{F},\Lambda}^{-\frac{1}{2}}F(x), \Lambda_{x} \pi_{F(x)}S_{\mathbf{F},\Lambda}^{-\frac{1}{2}},\omega  )$ is a Parseval continuous generalized fusion frame, applying Corollary \ref{thm: parsevalcgfframe}, we get
\begin{multline*}
\int _{X_1}\omega^2(x)\left\|\chi_x \pi_{V(x)} f\right\|^{2}d\mu(x)+\left\| \int _{X_1}\omega^2(x)\pi_{V(x)}\chi_x^*\chi_x \pi_{V(x)} fd\mu(x)\right\|^{2}\\
\int _{X_1^c}\omega^2(x)\left\|\chi_x \pi_{V(x)} f\right\|^{2}d\mu(x)+\left\| \int _{X_1^c}\omega^2(x)\pi_{V(x)}\chi_x^*\chi_x \pi_{V(x)} fd\mu(x)\right\|^{2}.
\end{multline*}
Moerever, we have 
\begin{eqnarray*}
\int _{X_1}\omega^2(x)\pi_{V(x)}\chi_x^*\chi_x \pi_{V(x)} f d\mu(x) &=&\int _{X_1}\omega^2(x)\left(\chi_x \pi_{V(x)}\right)^{*} \chi_x \pi_{V(x)} f d\mu(x) =\\
&=&\int _{X_1}\omega^2(x)\left(\Lambda_x \pi_{F(x)}S_{\mathbf{F},\Lambda}^{-\frac{1}{2}} \pi_{V(x)}\right)^{*} \Lambda_x \pi_{F(x)}S_{\Lambda}^{-\frac{1}{2}} \pi_{V(x)} f=\\
&=& \int _{X_1}\omega^2(x) S_{\mathbf{F},\Lambda}^{-\frac{1}{2}}  \pi_{V(x)} \Lambda_{x}^{*} \Lambda_{x}  \pi_{V(x)} S_{\mathbf{F},\Lambda}^{-\frac{1}{2}} f\\
&=& S_{\mathbf{F},\Lambda}^{-\frac{1}{2}} S_{\mathbf{F},\Lambda}S_{\mathbf{F},\Lambda}^{-\frac{1}{2}} f.
\end{eqnarray*}
Now, By replacing $f$ by $S_{\mathbf{F},\Lambda}^{\frac{1}{2}} f$, the proof is complete.
\end{proof}

\begin{cor}
	Let $(\Lambda, \ff,\omega)$ be a Parseval continuous generalized fusion frame for $\h$. Then 
	\begin{equation}
	0\leq S_{\mathbf{F},\Lambda}^{X_1}-S_{\mathbf{F},\Lambda}^{X_1}S_{\mathbf{F},\Lambda}^{-1}S_{\mathbf{F},\Lambda}^{X_1} \leq \frac{1}{4}S_{\mathbf{F},\Lambda}.
	\end{equation}
\end{cor}
\begin{proof} In the previous proof of theorem \ref{thm: ParsevalframeS}, we showed that
$$
\int _{X_1}\omega^2(x)\pi_{V(x)}\chi_x^*\chi_x \pi_{V(x)} f d\mu(x)= S_{\mathbf{F},\Lambda}^{-\frac{1}{2}} S_{\mathbf{F},\Lambda}S_{\mathbf{F},\Lambda}^{-\frac{1}{2}} f.
$$
By applying Corollary \ref{cor: parsevalcgfframe}, we get
$$
0 \leq \int _{X_1}\omega^2(x)\pi_{V(x)}\chi_x^*\chi_x \pi_{V(x)} f d\mu(x)-\left(\int _{X_1}\omega^2(x)\pi_{V(x)}\chi_x^*\chi_x \pi_{V(x)} f d\mu(x)\right)^{2} \leq \frac{1}{4} Id_{\h}
$$
Therefore, we have
$$
0 \leq S_{\mathbf{F},\Lambda}^{-\frac{1}{2}} \left( S_{\mathbf{F},\Lambda}^{X_1}- S_{\mathbf{F},\Lambda}^{X_1}S_{\mathbf{F},\Lambda}^{-1} S_{\mathbf{F},\Lambda}^{X_1}  \right) S_{\mathbf{F},\Lambda}^{-\frac{1}{2}}  \leq \frac{1}{4} Id_{\h}.
$$
\end{proof}
\begin{cor}
	Suppose that $(\Lambda, \ff,\omega)$ is  a continuous generalized fusion frame for $\h$ with continuous g-fusion frame operator $S_{\mathbf{F},\Lambda}$. If $X_1\subseteq X$ and $f\in\h$, we have
	\begin{equation}
	\int _{X_1}\omega^2(x) \Vert \Lambda_x\pi_{F(x)}(f)\Vert^2d\mu(x)- \Vert S_{\mathbf{F},\Lambda}^{-\frac{1}{2}} S_{\mathbf{F},\Lambda}^{X_1^c} f\Vert^2\geq\frac{3}{4}\left\| S_{\mathbf{F},\Lambda}^{-1}\right\|^{-1}\Vert  f\Vert^2.
   \end{equation}
\end{cor}
\begin{proof} By Theorem \ref{thm: ParsevalframeS} and Theorem \ref{thm: parsevalcgfframe}, we can write

\begin{align*}
&\int _{X_1}\omega^2(x) \left\| \Lambda_x\pi_{F(x)}(f) \right\|^{2} d\mu(x)+\left\| S_{\mathbf{F},\Lambda}^{-\frac{1}{2}} S_{\mathbf{F},\Lambda}^{X_1^c} f\right\|\\&=
\int _{X_1}\omega^2(x)\left\|\chi_x  \pi_{V(x)} S_{\mathbf{F},\Lambda}^{-\frac{1}{2}} f\right\|^{2} d\mu(x)+\left\|\int _{X_1^c}\omega^2(x) \pi_{V(x)} \chi_x^*\chi_x  f \pi_{V(x)}  S_{\mathbf{F},\Lambda}^{ \frac{1}{2}} f\right\|^{2}\\
& \geq \frac{3}{4}\left\| S_{\mathbf{F},\Lambda}^{\frac{1}{2}} f\right\|^{2}\\&=\frac{3}{4}\left\langle  S_{\mathbf{F},\Lambda} f, f\right\rangle \\&\geq \frac{3}{4}\left\| S_{\mathbf{F},\Lambda}^{-1}\right\|^{-1}\|f\|^{2} .
\end{align*}

The proof is complete.
\end{proof}
\begin{thm} Let $(\Lambda, \ff,\omega)$ be a  continuous generalized fusion frame for $\h$. Then for $X_1\subset X$, then,  for each $f\in \h$, we have 
	\begin{multline}
	\int _{X_1}\omega^2(x) \Vert \Lambda_x\pi_{F(x)}(f)\Vert^2 d\mu(x)-  \int _{X_1}\omega^2(x) \Vert \tilde{\Lambda}_{x}  \pi_{\tilde{F}(x)} \mathcal{M}_{\mathbf{F},\Lambda}^{X_1}
	(f) \Vert^2 d\mu(x)  \\=	\int _{X_1^c}\omega^2(x) \Vert \Lambda_x\pi_{F(x)}(f)\Vert^2d\mu(x)-  \int _{X_1^c}\omega^2(x) \Vert \tilde{\Lambda}_{x}  \pi_{\tilde{F}(x)} \mathcal{M}_{\mathbf{F},\Lambda}^{X_1^c} f\Vert^2 d\mu(x).
	\end{multline}
	Where \begin{align*}
	\mathcal{M}_{\mathbf{F},\Lambda}^{X_1}f=\int _{X_1}\omega^2(x)\pi_{F(x)} \Lambda_x^*\Lambda_x \pi_{F(x)} f.
	\end{align*}
\end{thm}
\begin{proof} Assume that  $S_{\mathbf{F},\Lambda}$ denote  the continuous generalized fusion frame for $(\Lambda, \ff,\omega)$, and by using the definition of $S_{\mathbf{F},\Lambda}$, it is clear that $\mathcal{M}_{\mathbf{F},\Lambda}^{X_1}+\mathcal{M}_{\mathbf{F},\Lambda}^{X_1^c}=S_{\mathbf{F},\Lambda} $. It follows that, $S_{\mathbf{F},\Lambda}^{-1}\mathcal{M}_{\mathbf{F},\Lambda}^{X_1 }+S_{\mathbf{F},\Lambda}^{-1} \mathcal{M}_{\mathbf{F},\Lambda}^{X_1^c}=Id_{\h}$. Hence, by applying Lemma \ref{lem: operatoroperation2} to the two operators $S_{\mathbf{F},\Lambda}^{-1}\mathcal{M}_{\mathbf{F},\Lambda}^{X_1 }$ and $S_{\mathbf{F},\Lambda}^{-1} \mathcal{M}_{\mathbf{F},\Lambda}^{X_1^c}$ yields that
$$
S_{\mathbf{F},\Lambda}^{-1} \mathcal{M}_{\mathbf{F},\Lambda}^{X_1}-S_{\mathbf{F},\Lambda}^{-1} \mathcal{M}_{\mathbf{F},\Lambda}^{X_1 }=\left(S_{\mathbf{F},\Lambda}^{-1} \mathcal{M}_{\mathbf{F},\Lambda}^{X_1}\right)^{2}-\left(S_{\mathbf{F},\Lambda}^{-1} \mathcal{M}_{\mathbf{F},\Lambda}^{X_1^c}\right)^{2} .
$$
Thus, for each $f, g \in \h$ we obtain
$$
\left\langle S_{\mathbf{F},\Lambda}^{-1} \mathcal{M}_{\mathbf{F},\Lambda}^{X_1} f, g\right\rangle-\left\langle S_{\mathbf{F},\Lambda}^{-1} \mathcal{M}_{\mathbf{F},\Lambda}^{X_1} S_{\mathbf{F},\Lambda}^{-1} \mathcal{M}_{\mathbf{F},\Lambda}^{X_1} f, g\right\rangle=\left\langle S_{\mathbf{F},\Lambda}^{-1} \mathcal{M}_{\mathbf{F},\Lambda}^{X_1^c} f, g\right\rangle-\left\langle S_{\mathbf{F},\Lambda}^{-1} \mathcal{M}_{\mathbf{F},\Lambda}^{X_1^c} S_{\mathbf{F},\Lambda}^{-1} \mathcal{M}_{\mathbf{F},\Lambda}^{X_1^{c}} f, g\right\rangle .
$$
one can choose $g$ to be $g=S_{\mathbf{F},\Lambda} f$, and we can get
$$
\left\langle \mathcal{M}_{\mathbf{F},\Lambda}^{X_1^c} f, f\right\rangle-\left\langle S_{\mathbf{F},\Lambda}^{-1} \mathcal{M}_{\mathbf{F},\Lambda}^{X_1} f, \mathcal{M}_{\mathbf{F},\Lambda} f\right\rangle=\left\langle \mathcal{M}_{\mathbf{F},\Lambda}^{X_1^c} f, f\right\rangle-\left\langle S_{\mathbf{F},\Lambda}^{-1} \mathcal{M}_{\mathbf{F},\Lambda}^{X_1^c} f, \mathcal{M}_{\mathbf{F},\Lambda}^{X_1^c} f\right\rangle .
$$
Finally, by \ref{eq: operator inverse}, the proof is complete.
\end{proof}
Observing that $(\frac{1}{\sqrt{\lambda}}\Lambda,\ff,\omega)$ is a  Parseval continuous generalized fusion frame if $(\Lambda, \ff,\omega)$ be a  $ \lambda$-tight continuous generalized fusion frame for $\h$. 
\begin{cor}\label{thm: tightcgfframe} Let $(\Lambda, \ff,\omega)$ be a  $ \lambda$-tight continuous generalized fusion frame for $\h$. Then for $X_1\subset X$ and $f\in \h$, the following conditions are fulfilled:
\begin{align}
0\leq \lambda \int _{X_1}\omega^2(x) \Vert \Lambda\pi_{F(x)}f\Vert^2d\mu(x) -\Vert \int _{X_1}\omega^2(x) \pi_{F(x)}\Lambda^*_x\Lambda_x  \pi_{F(x)}f d\mu(x)\Vert^2 \leq \frac{\lambda^2}{4}\Vert  f\Vert^2.
\end{align}

\begin{multline}
\frac{\lambda^2}{2}\Vert  f\Vert^2\leq \Vert \int _{X_1}\omega^2(x) \pi_{F(x)}\Lambda^*_x\Lambda_x  \pi_{F(x)}f d\mu(x)\Vert ^2d\mu(x)\Vert^2 \\- \Vert \int _{X_1^c}\omega^2(x) \pi_{F(x)}\Lambda^*_x\Lambda_x  \pi_{F(x)}f d\mu(x)\Vert^2  \leq  \frac{3\lambda^2}{2}\Vert  f\Vert^2.
\end{multline}
\begin{align}
\frac{3\lambda^2}{2}\Vert  f\Vert^2\leq \lambda \int _{X_1}\omega^2(x) \Vert \Lambda_x\pi_{F(x)}f\Vert^2d\mu(x)-\Vert \int _{X_1^c}\omega^2(x) \pi_{F(x)}\Lambda^*_x\Lambda_x  \pi_{F(x)}f d\mu(x)\Vert^2 \leq \lambda^2\Vert  f\Vert^2.
\end{align}
\end{cor}

Later we will discuss equality for tight continuous generalized fusion frames. With the purpose of doing this, first we define two operators $S_{\mathbf{F},\Lambda}^{1}$, $S_{\mathbf{F},\Lambda}^{2}$
\begin{eqnarray*}
S_{\mathbf{F},\Lambda}^{1}:\,\, \h \longrightarrow \h, \quad S_{\mathbf{F},\Lambda}^{1}f&=&\int _{X}a_x\omega^2(x) \pi_{F(x)}\Lambda^*_x\Lambda_x  \pi_{F(x)}(f) d\mu(x),\quad f\in\h.\\
S_{\mathbf{F},\Lambda}^{2}:\,\, \h \longrightarrow \h, \quad S_{\mathbf{F},\Lambda}^{2}f&=&\int _{X}(1-a_x)\omega^2(x) \pi_{F(x)}\Lambda^*_x\Lambda_x  \pi_{F(x)}(f) d\mu(x),\quad f\in\h.
\end{eqnarray*}
Where $(\Lambda, \ff,\omega)$ is a Bessel continuous generalized fusion frame for $\h$ and $\{a_x:\quad x\in X\}\in l^{\infty}(X)$, such that $ l^{\infty}(X)=\{\{ a_x:\quad x\in X\}:\quad\sup_{x\in X}\vert a_x\vert <\infty \}$.
\begin{prop}\label{prop: abounded linear operator} Let $(\Lambda, \ff,\omega)$ be a continuous $g$-fusion frame for $\h$ with bound $B$, then $S_{\mathbf{F},\Lambda}^{1}$, $S_{\mathbf{F},\Lambda}^{2}$ are bounded linear operators, and  
\begin{eqnarray}
 (S_{\mathbf{F},\Lambda}^{1})^*f &=&\int _{X}\overline{a}_x\omega^2(x) \pi_{F(x)}\Lambda^*_x\Lambda_x  \pi_{F(x)}(f) d\mu(x),\quad f\in\h.\\
 (S_{\mathbf{F},\Lambda}^{2})^*f &=&\int _{X}(1-\overline{a}_x)\omega^2(x) \pi_{F(x)}\Lambda^*_x\Lambda_x  \pi_{F(x)}(f) d\mu(x),\quad f\in\h.
\end{eqnarray}
\end{prop}
\begin{proof}
For $f\in\h$ and $X_1\subset X$, we have 
$$
\begin{aligned}
\left\|\int _{X}a_x\omega^2(x) \pi_{F(x)}\Lambda^*_x\Lambda_x  \pi_{F(x)}(f) d\mu(x)\right\|=& \sup _{g \in H,\|g\|=1}\left|\left\langle\int _{X}a_x\omega^2(x) \pi_{F(x)}\Lambda^*_x\Lambda_x  \pi_{F(x)}(f) d\mu(x), g\right\rangle\right| \\
=& \sup _{g \in H,\|g\|=1}\left|\int_{X_{1}} \omega^2(x)\left\langle\Lambda_x  \pi_{F(x)} f, \bar{a}_{x}\Lambda_x  \pi_{F(x)} g\right\rangle d \mu(x)\right| \\
\leq & \sup _{g \in H,\|g\|=1}\left(\int_{X_{1}}  \omega^2(x)\left\|\Lambda_x  \pi_{F(x)}(f)\right\|^{2} d \mu(x)\right)^{\frac{1}{2}} \\
& \times\left(\int_{X_{1}} v^{2}(x)\left\|\bar{a}_{x} \Lambda_x  \pi_{F(x)}(g)\right\|^{2} d \mu(x)\right)^{\frac{1}{2}} \\
\leq & B M_a\|f\|,
\end{aligned}
$$
where $M _a=\sup _{x \in X}\left|a_{x}\right|$ and $\bar{a}_{x}$ is the conjugate of $a_{x}$. which implies that $S_{\mathbf{F},\Lambda}^{1}$ is well-defined and $\left\|S_{\mathbf{F},\Lambda}^{1} f\right\| \leq B M_a\|f\|$. Therefore, $S_{\mathbf{F},\Lambda}^{1}$ is a bounded linear operator. Now let us compute its adjoint,
$$
\begin{aligned}
\left\langle f,\left(S_{\mathbf{F},\Lambda}^{1}\right)^{*}(g)\right\rangle=\left\langle S_{\mathbf{F},\Lambda}^{1} f, g\right\rangle &=\left\langle\int_{X} a_{x} \omega^2(x) \Lambda_x  \pi_{F(x)} f d \mu(x), g\right\rangle \\
&=\left\langle  f, \int _{X}\overline{a}_x \omega^2(x) \pi_{F(x)}\Lambda^*_x\Lambda_x  \pi_{F(x)}(g) d\mu(x)\right\rangle d \mu(x). \\
\end{aligned}
$$
Similarly, we can show that $S_{\mathbf{F},\Lambda}^{2}$ is a bounded linear operator and its adjoint is:
\begin{align*}
(S_{\mathbf{F},\Lambda}^{2})^*f =\int _{X}(1-\overline{a}_x)\omega^2(x) \pi_{F(x)}\Lambda^*_x\Lambda_x  \pi_{F(x)}(f) d\mu(x),\quad f\in\h.
\end{align*} 
\end{proof}
\begin{thm} Let $(\Lambda, \ff,\omega)$ be a $\lambda$-tight continuous generalized fusion frame for $\h$. Then for $f\in \h$ and $\{a_x:\quad x\in X\}\in l^{\infty}(X)$, we have 
	\begin{eqnarray*}
 \lambda\int _{X}a_x\omega^2(x) \Vert \Lambda\pi_{F(x)}(f)\Vert^2d\mu(x)-\Vert \int _{X}(1-a_x)\omega^2(x) \pi_{F(x)}\Lambda^*_x\Lambda_x  \pi_{F(x)}(f) d\mu(x)\Vert^2 \\
	= \lambda\int _{X}(1-\overline{a}_x)\omega^2(x) \Vert \Lambda\pi_{F(x)}(f)\Vert^2d\mu(x)-\Vert \int _{X}a_x\omega^2(x) \pi_{F(x)}\Lambda^*_x\Lambda_x  \pi_{F(x)}(f) d\mu(x)\Vert^2.
\end{eqnarray*}
Where $\overline{a}_x$ is the conjugate of $a_x$.
\end{thm}
\begin{proof}
Since  $(\Lambda, \ff,\omega)$ is a $\lambda$-tight continuous generalized fusion frame for $\h$, and by using Proposition \ref{prop: abounded linear operator}. In particular,  for each $f\in\h$, we have \begin{align*}
S_{\mathbf{F},\Lambda}^{1}f+S_{\mathbf{F},\Lambda}^{2}f=\int _{X} \omega^2(x) \pi_{F(x)}\Lambda^*_x\Lambda_x  \pi_{F(x)}(f) d\mu(x).
\end{align*} so $\lambda^{-1} S_{F}^{1}+\lambda^{-1} S_{\mathbf{F},\Lambda}^{2}=Id_{\h}$. Now if we suppose that $Q_{1}=\lambda^{-1} S_{\mathbf{F},\Lambda}^{1}$ and $Q_{2}=\lambda^{-1} S_{F}^{2}$, then, we have
$$
\begin{aligned}
Q_{1}+Q_{2}^{*} Q_{2} &=Q_{1}+\left(I_{H}-Q_{1}\right)^{*}\left(I_{H}-Q_{1}\right) \\
&=Q_{1}+\left(I_{H}-Q_{1}^{*}\right)\left(I_{H}-Q_{1}\right) \\
&=Q_{1}+I_{H}-Q_{1}-Q_{1}^{*}+Q_{1}^{*} Q_{1} \\
&=I_{H}-Q_{1}^{*}+Q_{1}^{*} Q_{1} \\
&=Q_{2}^{*}+Q_{1}^{*} Q_{1}
\end{aligned}
$$
and thus
$$
\lambda S_{\mathbf{F},\Lambda}^{1}+\left(S_{\mathbf{F},\Lambda}^{2}\right)^{*} S_{\mathbf{F},\Lambda}^{2}=\lambda S_{\mathbf{F},\Lambda}^{2}+\left(S_{\mathbf{F},\Lambda}^{1}\right)^{*} S_{\mathbf{F},\Lambda}^{1}.
$$
Hence for $h \in \h$, we get
$$
\begin{aligned}
\lambda & \int_{X} a_{x} v^{2}(x)\left\|\pi_{F(x)}(h)\right\|^{2} d \mu(x)+\left\|\int_{X}\left(1-a_{x}\right) v^{2}(x) \pi_{F(x)}(h) d \mu(x)\right\|^{2} \\
&=\left\langle\lambda S_{\mathbf{F},\Lambda}^{1} h, h\right\rangle+\left\langle\left(S_{\mathbf{F},\Lambda}^{2}\right)^{*} S_{\mathbf{F},\Lambda}^{2} h, h\right\rangle \\
&=\left\langle\left(\lambda S_{\mathbf{F},\Lambda}^{1}+\left(S_{\mathbf{F},\Lambda}^{2}\right)^{*} S_{\mathbf{F},\Lambda}^{2}\right) h, h\right\rangle \\
&=\left\langle\left(\lambda S_{\mathbf{F},\Lambda}^{2}+\left(S_{\mathbf{F},\Lambda}^{1}\right)^{*} S_{\mathbf{F},\Lambda}^{1}\right) h, h\right\rangle \\
&=\left\langle\lambda\left(S_{\mathbf{F},\Lambda}^{2}\right)^{*} h, h\right\rangle+\left\langle\left(S_{\mathbf{F},\Lambda}^{1}\right)^{*} S_{\mathbf{F},\Lambda}^{1} h, h\right\rangle \\
&=\left\langle h, S_{\mathbf{F},\Lambda}^{2} h\right\rangle+\left\|S_{\mathbf{F},\Lambda}^{1} h\right\|^{2}
\end{aligned}
$$
\end{proof}
Furthermore, by using Theorem \ref{thm: parsevalcgfframe} and Theorem  \ref{thm: tightcgfframe} we immediately obtain the following result:

\begin{cor}Let $(\Lambda, \ff,\omega)$ be a  $ \lambda$-tight continuous generalized fusion frame for $\h$. Then for $f\in \h$ and $\{a_x:\quad x\in X\}\in l^{\infty}(X)$, we have 
		\begin{eqnarray*}
			\lambda\int _{X}a_x\omega^2(x) \Vert \Lambda\pi_{F(x)}(f)\Vert^2d\mu(x)-\Vert \int _{X}(1-a_x)\omega^2(x) \pi_{F(x)}\Lambda^*_x\Lambda_x  \pi_{F(x)}(f) d\mu(x)\Vert^2 \\
			= \lambda\int _{X}(1-a_x)\omega^2(x) \Vert \Lambda\pi_{F(x)}(f)\Vert^2d\mu(x)-\Vert \int _{X}a_x\omega^2(x) \pi_{F(x)}\Lambda^*_x\Lambda_x  \pi_{F(x)}(f) d\mu(x)\Vert^2 \geq\frac{3}{4}\Vert  f\Vert^2.
		\end{eqnarray*}
\end{cor}
\section{Inequalities-Equalities for continuous generalized fusion pairs }
Applying Lemma \ref{lem: operatoroperation1}, we have 
\begin{align*}
\pi_{F(x)}S_{\mathbf{F},\Lambda}^{-1}= \pi_{F(x)} S_{\mathbf{F},\Lambda}^{-1}\pi_{ S_{\mathbf{F},\Lambda}^{-1}F(x)},
\end{align*}
implies that 
\begin{align*}
S_{\mathbf{F},\Lambda}^{-1}\pi_{F(x)}=\pi_{ S_{\mathbf{F},\Lambda}^{-1}F(x)}S_{\mathbf{F},\Lambda}^{-1} \pi_{F(x)}.
\end{align*}
Morever, \ref{eq: reconstruction formula} also can be rewritten as 
\begin{equation} \label{eq: reconstruction formulaaltenate1}
f=\int _{X}\omega^2(x) \pi_{S_{\mathbf{F},\Lambda}^{-1}F(x)} ( S_{\mathbf{F},\Lambda}^{-1}  \pi_{ F(x)}\Lambda_{x}^{*}  S_{\mathbf{F},\Lambda}) S_{\mathbf{F},\Lambda}^{-1}  \Lambda_{x} \pi_{ F(x)}fd\mu(x),\quad f\in \h.
\end{equation}
Now, we can introduce the following definition. 
\begin{defi} Let $\mathcal{V}=(\Lambda, \ff,\omega)$ be a continuous generalized fusion frame with bounds $ A$, $B$ and let $S_{\mathbf{F},\Lambda}$ be the frame operator. We consider also $\mathcal{W}=(\Gamma, \mathbf{G},\nu)$ a Bessel continuous generalized fusion mapping. We say that $\mathcal{W}$ is alternate dual of $\mathcal{V}$ if we have 
\begin{equation} \label{eq: reconstruction formulaaltenate2}
f=\int _{X}\omega(x)\nu(x)\pi_{G(x)} \Gamma_x^* S_{\mathbf{F},\Lambda}^{-1}\Lambda_{x}  \pi_{ F(x)} f d\mu(x),\quad f\in \h.
\end{equation}
\end{defi}
\begin{prop} The altenate dual of continuous generalized fusion frame of $\mathcal{V}$ is a continuous generalized fusion frame.
\end{prop}
\begin{proof} Applying \ref{eq: reconstruction formulaaltenate2}, for each $f\in\h$ we get 
\begin{align*}
\Vert f\Vert^2 &= \int _{X}\omega(x)\nu(x)\langle \pi_{G(x)} \Gamma_x^* S_{\mathbf{F},\Lambda}^{-1} \Lambda_{x} \pi_{ F(x)} f,f\rangle d\mu(x)\\ 
&\leq \int _{X}\omega(x)\nu(x)\langle  S_{\mathbf{F},\Lambda}^{-1}  \Lambda_{x} \pi_{ F(x)} f,\Gamma_x \pi_{G(x)} f\rangle d\mu(x)\\
&= \int _{X}\omega(x)\nu(x)\Vert S_{\mathbf{F},\Lambda}^{-1}  \Lambda_{x} \pi_{ F(x)} f\Vert \Vert \Gamma_x \pi_{G(x)} f\Vert
d\mu(x)\\
&\leq \left(\int _{X}\omega(x)^2\Vert S_{\mathbf{F},\Lambda}^{-1}  \Lambda_{x} \pi_{ F(x)} f\Vert^2 d\mu(x) \right)^{1/2} \left(\int _{X} \nu(x)^2 \Vert \Gamma_x \pi_{G(x)} f\Vert^2  d\mu(x)\right)^{1/2}\\
&\leq \Vert S_{\mathbf{F},\Lambda}^{-1}\Vert\sqrt{B}\left(\int _{X} \nu(x)^2 \Vert \Gamma_x \pi_{G(x)} f\Vert^2\right)^{1/2},
\end{align*}
where $B$ is the upper bound of $\mathcal{V}$.
\end{proof}
\begin{thm} Assume that  $(\Lambda, \ff,\omega)$ is a continuous $g$-fusion frame for $\h$ with the continuous $g$-fusion frame operator $S_{\mathbf{F},\Lambda}$, $(\Gamma, \mathbf{G},\nu)$ is the alternative continuous $g$-fusion frame of $(\Lambda, \ff,\omega)$. Then for any $X_1\subset X$ and for each  $f\in\h$,
\begin{multline}
\label{mult: th1}
\int _{X_1}\omega(x)\nu(x) \langle S_{\mathbf{F},\Lambda}^{-1}\Lambda_x \pi_{F(x)} f ,\Gamma_x \pi_{G(x)}f\rangle d\mu(x)-\left \Vert \int _{X_1}\omega(x)\nu(x)\pi_{G(x)} \Gamma_x^* S_{\mathbf{F},\Lambda}^{-1}\Lambda_x \pi_{F(x)}f d\mu(x)\right\Vert^2 \\=\int _{X_1^c}\omega(x)\nu(x) \overline{\langle S_{\mathbf{F},\Lambda}^{-1}\Lambda_x \pi_{F(x)} f ,\Gamma_x \pi_{G(x)}f\rangle} d\mu(x)- \left\Vert \int _{X_1^c}\omega(x)\nu(x) \pi_{G(x)} \Gamma_x^{*}  S_{\mathbf{F},\Lambda}^{-1} \Lambda_{x}\pi_{ F(x)} f d\mu(x)\right\Vert^2.
\end{multline}
\end{thm}
\begin{proof}
For each $X_1\subset X$, let us consider a bounded linear operator $\mathcal{T}_{\mathbf{FG},\Lambda \Gamma}$ as follows
\begin{align*}
\mathcal{T}_{\mathbf{FG},\Lambda \Gamma}^{X_1} f= \int _{X_1}\omega(x)\nu(x)\pi_{G(x)} \Gamma_x ^{*}S_{\mathbf{F},\Lambda}^{-1} \Lambda_{x}\pi_{ F(x)} f d\mu(x), \quad  f\in \h.
\end{align*}
 It is clearly that $ \mathcal{T}_{\mathbf{FG},\Lambda \Gamma}^{X_1}+\mathcal{T}_{\mathbf{FG},\Lambda \Gamma}^{X_1^c}= Id_{\h}$. Applying Lemma \cite{th: theorem1}.
 \begin{align*}
 & \int _{X_1}\omega(x)\nu(x) \langle _{\mathbf{F},\Lambda}^{-1}\Lambda_x \pi_{F(x)} f, \Gamma_x \pi_{G(x)}f\rangle d\mu(x)-\left \Vert \int _{X_1}\omega(x)\nu(x)\pi_{G(x)} \Gamma_x ^{*}S_{\mathbf{F},\Lambda}^{-1} \Lambda_{x}\pi_{ F(x)} f d\mu(x)\right\Vert^2\\
 &=\int _{X_1}\omega(x)\nu(x) \langle _{\mathbf{F},\Lambda}^{-1}\Lambda_x \pi_{F(x)} f, \Gamma_x \pi_{G(x)}f\rangle d\mu(x)- \langle \mathcal{T}_{\mathbf{FG},\Lambda \Gamma}^{X_1}f, \mathcal{T}_{\mathbf{FG},\Lambda \Gamma}^{X_1}f\rangle \\ &=\langle \mathcal{T}_{\mathbf{FG},\Lambda \Gamma}^{X_1}f,f\rangle +\langle (\mathcal{T}_{\mathbf{FG},\Lambda \Gamma}^{X_1})^{\ast}\mathcal{T}_{\mathbf{FG},\Lambda \Gamma}^{X_1}f, f\rangle \\
 &=\langle (\mathcal{T}_{\mathbf{FG},\Lambda \Gamma}^{X_1^c})\ast f,f\rangle +\langle \mathcal{T}_{\mathbf{FG},\Lambda \Gamma}^{X_1^c}f, \mathcal{T}_{\mathbf{FG},\Lambda \Gamma}^{X_1^c}f\rangle \\
&=\langle f,  \int _{X_1^c}\omega(x)\nu(x)\pi_{G(x)} \Gamma_x ^{*}S_{\mathbf{F},\Lambda}^{-1} \Lambda_{x}\pi_{ F(x)} f d\mu(x)  \rangle -\left\Vert \int _{X_1}\omega(x)\nu(x)\pi_{G(x)} \Gamma_x ^{*}S_{\mathbf{F},\Lambda}^{-1} \Lambda_{x}\pi_{ F(x)} f d\mu(x)\right\Vert^2\\
&=\int _{X_1^c}\omega(x)\nu(x)\langle \Gamma_x \pi_{G(x)}f ,S_{\mathbf{F},\Lambda}^{-1}\Lambda_x \pi_{F(x)}f\rangle d\mu(x)- \left\Vert \int _{X_1^c}\omega(x)\nu(x)\pi_{G(x)} \Gamma_x ^{*}S_{\mathbf{F},\Lambda}^{-1} \Lambda_{x}\pi_{ F(x)} f  d\mu(x)\right\Vert^2.
\end{align*}
Which complies the proof.
\end{proof}
In the case of Parseval fusion frame, the previous equality can have a special form  as follows: 
\begin{cor}
Let  $(\Lambda, \ff,\omega)$ be a Parseval continuous $g$-fusion frame for $\h$ with the continuous $g$-fusion frame operator $S_{\mathbf{F},\Lambda}$, $(\Gamma, \mathbf{G},\nu)$ is the alternative continuous $g$-fusion frame of $(\Lambda, \ff,\omega)$. Then for any $X_1\subset X$ and for each  $f\in\h$,
\begin{multline}\label{mult: th1}
 \int _{X_1}\omega(x)\nu(x) \langle \Lambda_x \pi_{F(x)} ,\Gamma_x \pi_{G(x)}(f)\rangle d\mu(x)- \Vert \int _{X_1}\omega(x)\nu(x)\pi_{G(x)} \Gamma_x ^{*}S_{\mathbf{F},\Lambda}^{-1} \Lambda_{x}\pi_{ F(x)} f d\mu(x)\Vert^2\\= \int _{X_1^c}\omega(x)\nu(x) \langle \Lambda_x \pi_{F(x)} ,\Gamma_x \pi_{G(x)}(f)\rangle d\mu(x)- \left\Vert \int _{X_1^c}\omega(x)\nu(x)\pi_{G(x)} \Gamma_x ^{*} \Lambda_{x}\pi_{ F(x)} f d\mu(x)\right\Vert^2.
\end{multline}
\end{cor}
\section{Frame operator of a pair of Bessel continuous generalized fusion mappings }
From now, let us consider two  Bessel continuous generalized fusion mappings: $\mathcal{V}=(\Lambda, \ff,\omega)$ with Bessel bound $B_1$ and $\mathcal{W}=(\Gamma, \mathbf{G},\nu)$ with Bessel bound $B_2$. We define the operator 
\begin{align}
S_{\mathbf{FG},\Lambda \Gamma}(f) =\int _{X}\omega(x)\nu(x) \pi_{F(x)}\Lambda^*_x \Gamma_x  \pi_{G(x)}(f) d\mu(x),\quad f\in\h.
\end{align}
For all $f,g\in\h$, we have also
\begin{align*}
\langle S_{\mathbf{FG},\Lambda \Gamma}f,g\rangle = \int _{X}\omega(x)\nu(x)\langle \Gamma_x  \pi_{G(x)}f,\Lambda _x\pi_{F(x)}g\rangle d\mu(x).
\end{align*}
Furthermore, by using Cauchy-Schwartz inequality, we have 
\begin{equation}\label{eq: Besselpairs}
\vert \langle S_{\mathbf{FG},\Lambda \Gamma}f,g\rangle\vert \leq\left(\int _{X}\omega^2(x)\Vert \Gamma_x  \pi_{G(x)}f\Vert^2  d\mu(x)\right)^{1/2} \left(\int _{X} \nu^2(x)\Vert\Lambda _x\pi_{F(x)}g\Vert^2  d\mu(x)\right)^{1/2}.
\end{equation}
From (\ref{eq: Besselpairs}), it follows
\begin{align*}
\vert \langle S_{\mathbf{FG},\Lambda \Gamma}f,g\rangle\vert \leq\sqrt{B_1}\sqrt{B_2}\Vert g\Vert\Vert f\Vert.
\end{align*}
Hence, $ S_{\mathbf{FG},\Lambda \Gamma}$is a bounded operator and we have 
$$ \Vert S_{\mathbf{FG},\Lambda \Gamma}\Vert \leq \sqrt{B_1}\sqrt{B_2}. $$
From (\ref{eq: Besselpairs}), we obtain also
\begin{equation}\label{eq: besselpairlamdagama}
\Vert  S_{\mathbf{FG},\Lambda \Gamma}f \Vert \leq\sqrt{B_1}\left(\int _{X} \nu^2(x)\Vert\Lambda _x\pi_{F(x)}g\Vert^2  d\mu(x)\right)^{1/2}.
\end{equation}
and 
\begin{equation}
\Vert  (S_{\mathbf{FG},\Lambda \Gamma})^*f \Vert\leq\sqrt{B_2}\left(\int _{X}\omega^2(x)\Vert \Gamma_x  \pi_{G(x)}f\Vert^2  d\mu(x)\right)^{1/2} 
\end{equation}
Morever, from the adjointability of the operator $S_{\mathbf{FG},\Lambda \Gamma}$, we get 
\begin{align*}
\langle S_{\mathbf{FG},\Lambda \Gamma}f,g\rangle= \int _{X} \omega(x)\nu(x)\langle  f, \pi_{F(x)}\Gamma_x^*\Lambda _x\pi_{G(x)}g\rangle d\mu(x).
\end{align*}
Hence, $S_{\mathbf{FG},\Lambda \Gamma}^*=S_{\mathbf{GF},\Gamma\Lambda } $.
\begin{thm} The  following assertions are equivalent:
\begin{itemize}
\item[i)] $ S_{\mathbf{FG},\Lambda \Gamma}$  is  bounded below;
\item[ii)] There exist $K\in \mathcal{B}(\h)$ such that $\{T_x\}_{x\in X}$ is a resolution of identity, where
\begin{align}
T_x=\omega(x)\nu(x) K \pi_{F(x)}\Lambda^*_x \Gamma_x  \pi_{G(x)},\quad x\in X.
\end{align}
If one of conditions is satisfied, then $\mathcal{W}$ is a continuous generalized fusion frame.
\end{itemize}
\end{thm}
\begin{proof} $(i)\Rightarrow (ii)$ Obvious.\\
$(ii)\Leftarrow (i)$ If $(ii)$ holds, then for $f,g\in\h$, we have 
$$
\begin{aligned}
\left\langle K S_{\mathbf{FG},\Lambda \Gamma} f, g\right\rangle=\left\langle S_{\mathbf{FG},\Lambda \Gamma} f,K^{*} g\right\rangle &=  \int _{X} \omega(x)\nu(x)\left\langle f,\left( K \pi_{F(x)}\Lambda^*_x\Gamma_x   \pi_{G(x)}\right)^{*} g\right\rangle d \mu(x) \\
&=\langle f, g\rangle
\end{aligned}
$$
which implies $I_{\h}=K S_{\mathbf{FG},\Lambda \Gamma}$. Thus, $ S_{\mathbf{FG},\Lambda \Gamma}$  is  bounded below.\\
Now, if $S_{\mathbf{FG},\Lambda \Gamma}$ is bounded below, from \ref{eq: besselpairlamdagama} it follows that $\mathcal{G}$ is a continuous fusion frame.

\begin{align}
f=\omega(x)\nu(x) K \pi_{F(x)}\Lambda^*_x \Gamma_x  \pi_{G(x)}f d\mu(x),\quad x\in X.
\end{align}
Hence
\begin{align}
f=K \left( \int _{X}\omega(x)\nu(x)  \pi_{F(x)}\Lambda^*_x \Gamma_x  \pi_{G(x)}f d\mu(x)\right),\quad x\in X.
\end{align}

\end{proof}
\begin{cor} The  following assertions are equivalent:
\begin{itemize}
\item[i)] $ S_{\mathbf{FG},\Lambda \Gamma}$ is invertible operator;
\item[ii)] There exist $K\in \mathcal{B}(\h) $  invertible such that $\{T_x\}_{x\in X}$ is a resolution of identity, where
\begin{align*}
T_x=\omega(x)\nu(x) K \pi_{F(x)}\Lambda^*_x \Gamma_x  \pi_{G(x)},\quad x\in X \\is \,\,a\,\, resolution \,\,of \,\,identity.
\end{align*}
If one of conditions is satisfied, then $\mathcal{V}$, $\mathcal{W}$ are continuous generalized fusion frames.
\end{itemize}
\end{cor}
\begin{thm}\label{th: paircgfframeeq} We assume there exist $\lambda_1<1$, $\lambda_2>-1$ such that 
\begin{align*}
\left\Vert f-\int _{X}\omega(x)\nu(x) \pi_{F(x)}\Lambda^*_x \Gamma_x  \pi_{G(x)}(f) d\mu(x)\right\Vert\leq \lambda_1\Vert f\Vert+\lambda_2\left\Vert  \int _{X}\omega(x)\nu(x) \pi_{F(x)}\Lambda^*_x \Gamma_x  \pi_{G(x)}(f) d\mu(x)\right\Vert,
\end{align*}
for any $f\in\h$. Then $\mathcal{W}$ is continuous generalized fusion frame and 
\begin{align*}
\left( \dfrac{1-\lambda_1}{1+\lambda_2}\right)^2\dfrac{1}{B_1}\Vert f\Vert^2\leq \int _{X}\omega^2(x)\Vert \Gamma_x  \pi_{G(x)}f\Vert^2 d\mu(x), \quad f\in \h.
\end{align*}
\end{thm}
\begin{proof} Since, $S_{\mathbf{FG},\Lambda \Gamma}f=\int _{X}\omega(x)\nu(x) \pi_{F(x)}\Lambda^*_x \Gamma_x  \pi_{G(x)}(f) d\mu(x)$.\\
We have\begin{align*}
\Vert f-S_{\mathbf{FG},\Lambda \Gamma}f\Vert\leq\lambda_1\Vert f\Vert+\lambda_2\Vert S_{\mathbf{FG},\Lambda \Gamma}f\Vert,
\end{align*}
therefore, since \begin{align*}
\Vert f-S_{\mathbf{FG},\Lambda \Gamma}f\Vert\geq \mid\vert \Vert f\Vert-\Vert S_{\mathbf{FG},\Lambda \Gamma}f\Vert \mid,
\end{align*}
it follows
\begin{align*}
\lambda_1\Vert f\Vert+\lambda_2\Vert S_{\mathbf{FG},\Lambda \Gamma}f\Vert\geq \vert \Vert f\Vert-\Vert S_{\mathbf{FG},\Lambda \Gamma}f\Vert, 
\end{align*}
thus, \begin{align*}
\Vert S_{\mathbf{FG},\Lambda \Gamma}f\Vert\geq\dfrac{1-\lambda_1}{1+\lambda_2}\Vert f\Vert.
\end{align*}
By using , we obtain 
\begin{align*}
 \int _{X}\omega^2(x)\Vert \Gamma_x  \pi_{G(x)}f\Vert^2 d\mu(x)  \geq\left( \dfrac{1-\lambda_1}{1+\lambda_2}\right)^2\dfrac{1}{B_1}\Vert f\Vert^2.
\end{align*}
\end{proof}
In particular, if we take $\lambda_2=0$, then in this case we have obviously a stronger result.
\begin{cor} We assume that exists $\lambda\in [0,1)$, such that  
\begin{equation}
\label{eq: cgfframepairs}
\left\Vert f-\int _{X}\omega(x)\nu(x) \pi_{F(x)}\Lambda^*_x \Gamma_x  \pi_{G(x)}(f) d\mu(x)\right\Vert\leq \lambda\Vert f\Vert,\quad f\in \h.
\end{equation}
Then $\mathcal{V}$, $\mathcal{W}$ are continuous generalized fusion frames and the following estimates hold
\end{cor}
\begin{proof} Under the assumption (\ref{eq: cgfframepairs}),  for each $f\in\h$, we have
\begin{eqnarray}
\left\Vert f- S_{\mathbf{FG},\Lambda \Gamma}f\right\Vert &= &\left\Vert ( I_{\h}- S_{\mathbf{FG},\Lambda \Gamma})^*f\right\Vert \leq\left\Vert ( I_{\h}- S_{\mathbf{FG},\Lambda \Gamma})^* \right\Vert \left\Vert f\right\Vert \\&\leq& \lambda \left\Vert f\right\Vert
\end{eqnarray}

Hence applying Theorem \ref{th: paircgfframeeq}, one arrives at the conclusion of the Corollary. 
\end{proof}
\textbf{ Conclusions}
\\
 In this paper, we have established  some equalities and inequalities for continuous generalized fusion frame, Parseval continuous generalized fusion frame, alternate dual continuous generalized fusion frame, which generalize some remarkable and existing results which have been obtained.\\
 
\bibliographystyle{amsplain}

\end{document}